\documentclass[11pt,a4paper]{article}
\usepackage{hyperref}
\usepackage[pdftex]{graphicx}
\usepackage{amssymb,amsthm,amsmath,hyperref,txfonts}

\usepackage{mathtools}

\usepackage[dvips]{color}
\usepackage{graphicx,float,xcolor}
\usepackage{epstopdf}
\usepackage{tikz}
\usetikzlibrary{positioning,decorations,arrows,patterns,shapes,backgrounds}

\newtheorem{theorem}{Theorem}[section]

\newtheorem{definition}{Definition}[section]
\newtheorem{lemma}[theorem]{Lemma}

\newcommand{\RR}{\mathbb{R}}

\newcommand{\Dt}{\dfrac{d}{dt}}

\newcommand{\intO}{\int_{\Omega}}

\DeclareMathOperator{\curl}{curl}

\DeclareMathOperator{\divv}{div}
\newcommand{\Ss}{\mathbf{S}}

\newcommand{\uu}{\mathbf{u}}

\DeclareMathOperator{\D}{{\mathbb D}}

\newcommand{\del}{\partial}

\def\XXint#1#2#3{{\setbox0=\hbox{$#1{#2#3}{%
\int}$ }
\vcenter{\hbox{$#2#3$ }}\kern-.6\wd0}}

\def\edc{\end{document}}
\numberwithin{equation}{section}

\date{\today}

\allowdisplaybreaks

\begin{document}
\title{\bf Blow-up criterion for a three-dimensional compressible non-Newtonian fluid with vacuum
}

\author{Guo Junyuan\thanks{E-mail: 202220772@stumail.nwu.edu.cn}, Fang Li$^a$\thanks{E-mail: fangli@nwu.edu.cn}\thanks{Corresponding author} \\
\textit{\small Department of Mathematics and CNS, Northwest University, Xi'an, P. R. China}
}
\date{}
\maketitle
\begin{abstract}
		This work is devoted to establish an improved blow-up criterion for strong solutions to a three-dimensional compressible non-Newtonian fluid with vacuum. The considered system is the Power Law model in a bounded periodic domain in ${\mathbf R}^3$. We establish a blow-up criterion for the local strong solutions in terms of the $L^4(0,T;L^\infty(\Omega))-$norm of the gradient of the velocity for any power-law index $q>1$. 
\vspace{4mm}
\end{abstract}

\section{Introduction}
In this paper, we analyze a initial-boundary value problem for the Navier-Stokes equations describing the flow of a compressible non-Newtonian fluid that reads
\begin{equation}\label{1dlg-E2}
	\begin{cases}
		\partial_{t}\rho+\operatorname{div}(\rho \uu)=0\quad in~\Omega_{T},\\
		\partial_{t}(\rho \uu)+\operatorname{div}(\rho \uu\otimes \uu)-\operatorname{div}\Ss+\nabla p=0\quad in~\Omega_{T}.
	\end{cases}
\end{equation}
Here $\Omega\subset\mathbf{R}^{3}$ is the domain occupied by the fluid, $T>0$ is the time of evolution and $Q_{T}=(0,T)\times \Omega$. We denote by $u(t,x)$ the fluid velocity, by $\rho(t,x)$ the fluid density, by $p$ is the pressure which is assumed here to be a given function of $\rho$, i.e. $p=a\rho^\gamma,$ with $a>0$ and $\gamma>1$, and by $\Ss$ the viscous stress. We assume the periodic boundary condition is a torus, that is, $$ \Omega={\mathbb T}^{3}.$$

We restrict ourselves to the constitutive relation
\begin{equation}\label{S}
	\Ss=2\mu{\mathbb D}{\mathbf u}+\lambda\operatorname{div}\uu~\mathbb{I}+\tau^{*}(|{\mathbb D}{\mathbf u}|^{2}+\delta^{2})^{\frac{q-2}{2}}{\mathbb D}{\mathbf u},
\end{equation}
with $q$ is a number belongs to $(1,\infty)$, $\tau^*\geqslant 0$ is the yield stress also assumed to be a constant, shear rate ${\mathbb D}{\mathbf u} =\frac{1}{2}(\nabla \uu+\nabla^t \uu)$, and $|{\mathbb D}{\mathbf u}|^2$ is the Hilbert-Schmidt norm defined by (also called Frobenius norm)
$$|{\mathbb D}{\mathbf u}|^2=\underset{i,j=1}{\overset{3}{\sum}}|{\mathbb D}_{ij}{\mathbf u}|^2$$
with ${\mathbb D}{\mathbf u} =({\mathbb D}_{ij}{\mathbf u})_{1\leqslant i,j\leqslant 3}$ and ${\mathbb D}_{ij}{\mathbf u}=\frac{1}{2}(\partial_i\uu_j+\partial_j\uu_i).$ The coefficients $\mu$ and $\lambda$ are the so-called Lam\'e viscosity coefficients assumed here constants, while they satisfy the following physical restrictions:
\begin{equation}\label{constant}
	\mu>0\quad2\mu+\lambda>0\quad\delta\geqslant c>0.
\end{equation}	
Consider the initial value
\begin{equation}\label{initial value}
	(\rho,\uu)(x,0)=(\rho_{0},\uu_{0})(x)\quad in\quad \Omega.
\end{equation}

The earliest work of incompressible non-Newtonian fluid is attributed to Ladyzhenskaya  \cite{Ladyzhenskaya-1969, Ladyzhenskaya-1970}.
Ladyzhenskaya gave  $\Ss=\nu_1 {\D}{\uu}+\nu_2|{\mathbb D}{\mathbf u}|^{r-2}{\D}{\uu}$  with $\nu_1\geqslant 0$ and  $\nu_2>0$ being constants and
studied the global existence of weak solutions for Dirichlet boundary conditions with the exponent $r>1+\frac{2d}{d+2}$
($d$ stands for space dimension). Inspired by \cite{Ladyzhenskaya-1969, Ladyzhenskaya-1970}, the non-Newtonian fluids has been studied intensively, and various existence and regular properties have been proved in the last years. Bellout, Bloom and Ne\v{c}as \cite{Bellout-1994} proved that there exist Young-measure-valued solutions to the incompressible non-Newtonian fluids for space periodic problems under some conditions.
Ne\v{c}asov\'{a} and Penel \cite{Necasova-2001} proved $L^2-$decay for weak solution to an incompressible non-Newtonian fluid in whole space
under some assumptions. Guo and his cooperators obtained a series results about incompressible non-Newton fluids (see the  monographs
\cite{Guo-Zhu-2002,Guo-Lin-Shang-2006}).  Bart{\l}omiej and Aneta  showed the existence of weak solutions for unsteady flow of incompressible
nonhomogeneous, heat-conducting fluids with generalised form of the stress tensor without any restriction on its upper growth
\cite{Bartlomiej-Aneta-2015}. More results on the mathematical theory of the incompressible non-Newtonian fluids, we can refer the monographs
\cite{Chhabra-2008, Malek-1997} and  papers \cite{Bellout-1994,Bohme-1987,Diening-2010,Guo-Zhu-2002,Guo-Lin-Shang-2006,Zhikov-2009,Zang-2018} and therein references.

The first mathematical analysis on the compressible non-Newtonian fluid goes back to  \cite{Necasova-1993}, where the existence of the measure-valued solutions to
the equations (\ref{1dlg-E2})-(\ref{S}) was proved. Later, Mamontov \cite{Mamontov-1999,Mamontov-1999-1} illustrated the global existence of  weak solution  under the assumptions of an exponentially growing viscosity and isothermal pressure. Recently, Abbatiello, Feireisl and Novot\'{n}y proved the existence of so-called dissipative solution in \cite{Abbatiello-2020}. For the existence of the strong solution, the local-in-time existence were established for the absence of vacuum in \cite{Kalousek-Macha-Necasova-2020}. Xu and Yuan \cite{Yuan-Xu-2008} proved the local-in-time existence and uniqueness in one space dimension with singularity and vacuum. While the initial energy is small, Yuan, Si and Feng \cite{Yuan-Si-Feng-2019} established the global well-posedness  of strong solutions for the initial boundary value problems of the one-dimensional model (\ref{1dlg-E2})-(\ref{S}).  Fang and Zang showed the global existence and uniqueness of strong solutions to the Cauchy problem for a one-dimensional compressible non-Newtonian fluid of power-law type in \cite{Fang-Zang-2023}.  	

Recently, Bilal Al Taki \cite{2023Wellposedness} studied the local existence and uniqueness of the strong solution of
system (\ref{1dlg-E2})-(\ref{S}) in three-dimensional space. It is a nature question that whether the strong solution of
system (\ref{1dlg-E2})-(\ref{S}) blows up in finite time.

It is vital to review the abundant results of blow up criteria of the compressible Navier-Stokes equations since they will guide us to hunt the blow
up criteria of compressible non-Newtonian fluid. The famous result of the blow up criterion of the compressible Navier-Stokes equations was proved by
Fan and Jiang in \cite{2008BLOW} for two dimensions and $7\mu>9\lambda$ that
$$\lim\limits_{T\rightarrow T^{*}}\left(\sup_{0\leqslant t\leqslant
	T}\|\rho\|_{L^{\infty}(\Omega)}+\int_{0}^{T}(\|\rho\|_{W^{1,q_{0}}(\Omega)}+\|\nabla\rho\|_{L^{2}(\Omega)}^{4})dt\right)=\infty.$$
Later, Huang and Xin \cite{2011Serrin} proved the following blow-up criterion to the compressible Navier-Stokes equations
$$\lim\limits_{T\rightarrow T^{*}}
\left(\|\divv \uu\|_{L^{1}(0,T;L^{\infty}(\mathbb{R}^{3}))}+\|\rho^{\frac{1}{2}}\uu\|_{L^{s}(0,T;L^{r}(\mathbb{R}^{3}))}\right)=\infty,$$                        
and
$$\lim\limits_{T\rightarrow
	T^{*}}\left(\|\rho\|_{L^{\infty}(0,T;L^{\infty}(\mathbb{R}^{3}))}+\|\rho^{\frac{1}{2}}\uu\|_{L^{s}(0,T;L^{r}(\mathbb{R}^{3}))}\right)=\infty,$$
with $\frac{2}{s}+\frac{3}{r}\leqslant 1 ~( 3<r\leqslant\infty).$
Fang, Song and Guo \cite{2013A} gave the blowup criterion of
compressible non-Newtonian fluid equations over one-dimensional bounded interval as follows
$$\lim\limits_{T\rightarrow
	T^{*}}sup\left(\|\rho\|_{L^{\infty}(0,T;H^{1}(I))}+\| \uu_{x}\|_{L^{\infty}(0,T;L^{p}(I))}\right)=\infty.$$                                          
Yuan and his cooperators proved in \cite{2014Blow} that
$$\lim\limits_{T\rightarrow T^{*}}\int_{0}^{T}\|\uu_{x}\|_{L^{\infty}(I)}dt=\infty,$$
is a blowup criterion for compressible non-Newtonian fluids in one-dimensional bounded interval. Very recently,
Bilal Al Taki and his cooperators \cite{2023remark} studied the blowup criterion for compressible non-Newtonian fluids in three dimensions, and proved the following blow up criteria that
$$\lim\limits_{T\rightarrow T^{*}}sup\left(\|\rho\|_{W^{1,q_{0}}(\Omega)}+\|\uu\|_{H^{1}(\Omega)}\right)=\infty,$$
with $q\in(3,\infty),q_{0}=\min{\{6,q\}}.$

\begin{definition}\label{def-sol}
	The pair $(\rho,\uu)$ is called a strong solution to system \eqref{1dlg-E2}-\eqref{S} if $(\rho,\uu)$ is a weak solution, satisfies equations \eqref{1dlg-E2} almost everywhere in $(0,T^*)\times \Omega$, and enjoys the following properties
	\begin{equation}\label{1dlg-E6}
		\begin{aligned}
			&\rho\in L^{\infty}(0,T^*; W^{1,6}(\Omega))  \hspace*{1.7cm}(\rho)_t \in L^\infty(0,T^*; L^6(\Omega))\\
			&\uu\in L^{\infty}(0,T^*; H^{1}(\Omega)) \hspace*{2cm} \nabla^2 \uu\in L^{2}(0, T^*; L^6(\Omega))\\
			&\sqrt{\rho}\uu_t\in L^{\infty}(0, T^*; L^2(\Omega)) \hspace*{1.5cm}  \uu_t\in L^{2}(0, T^*; H^{1}(\Omega)).
		\end{aligned}
	\end{equation}
\end{definition}

The aim of the present paper is to give the blow-up criterion to the system \eqref{1dlg-E2}-\eqref{initial value}.
Our main results are the following.

\begin{theorem}\label{theo1}
	Let $\Omega={\mathbb T}^3$ be a periodic domain and
	$$\mu>0\quad2\mu+\lambda>0\quad\delta\geqslant c>0.$$
	Suppose that the initial data $(\rho_{0},\uu_{0}) $ satisfy
	\begin{equation}\label{1dlg-E7}
		0\leqslant\rho_{0}\in W^{1,6}(\Omega),\quad \uu_{0}\in W^{2,6}(\Omega),
	\end{equation}
	and the compatibility condition:
	\begin{equation}\label{1dlg-E8}
		-\operatorname{div}\left(2\mu {\mathbb D}{\mathbf u}_0+\lambda\operatorname{div} \uu_{0}~\mathbb{I}+\tau^{*}(|{\mathbb D}{\mathbf u}_0|^{2}+\delta^{2})^{\frac{q-2}{2}}{\mathbb D}{\mathbf u}_0\right)+\nabla p_{0}=\sqrt{\rho_{0}}g,
	\end{equation}
	where $g $ is a function in $L^{6}(\Omega)$.
	Assume that $(\rho,\uu)$ is a local strong solution to the initial-boundary-value problem \eqref{1dlg-E2}-\eqref{initial value} satisfying \eqref{1dlg-E7}-\eqref{1dlg-E8} on $(0,T^*)\times \Omega$ for the maximal time of existence $T^*>0.$ Then
	\begin{equation}\label{1dlg-E9}
		\lim\limits_{T\rightarrow T^{*}}\int_{0}^{T}\|\nabla \uu\|_{L^{\infty}(\Omega)}^{4}dt=\infty.
	\end{equation}
\end{theorem}

\begin{theorem}\label{theo2}
	Under the conditions of Theorem \ref{theo1},~the following is also true:
	\begin{equation}\label{1dlg-E10}
		\lim\limits_{T\rightarrow T^{*}}\left(\|\rho\|_{L^{\infty}(0,T;L^{\infty}(\Omega))}+\|\nabla \uu\|_{L^{\infty}(0,T;L^{3}(\Omega))}\right)=\infty.
	\end{equation}
\end{theorem}

We now comment on the analysis of this paper. In the proof of our main theorems, we shall deal with several difficulties. First, it is not easy to get higher order estimation by direct calculation, so we use the iterative method. Second, we shall use the results established in Lemma \ref{lem3} to estimate $\|\uu\|_{W^{2,6}(\Omega)}$ .

The rest of the paper is organized as follows: In Section 2, we collect some elementary
facts and inequalities which will be needed later. The main results, Theorem \ref{theo1} and Theorem \ref{theo2} proved in Section 3 and Section 4 respectively.  Notice that in all the estimates established below, we will denote by $C$ a generic constant may depending on $a, \mu, \lambda, \varepsilon$ and the norms of the data, however, it does not depend on the parameter $\delta$.

\section{Preliminaries}\label{Sec-P}

In this section, we begin with the local existence and uniqueness of strong solutions obtained in \cite{2023Wellposedness}.

\begin{lemma}\label{lem1}
	If the initial data $(\rho_{0},\uu_{0})$ satisfy \eqref{1dlg-E7} and \eqref{1dlg-E8}. then there exists a small time $T_1>0$ and a unique strong solution $(\rho,\uu)$ to the initial-boundary-value problem \eqref{1dlg-E2}-\eqref{initial value} in $\Omega\times(0,T_{1})$.
\end{lemma}

Next, the well-known Gagliardo-Nirenberg inequality which will be used later frequently (see \cite{1968Solonnikov}).

\begin{lemma}\label{lem2}
	{\upshape(\textbf{Gagliardo-Nirenberg inequality})}
	Let $\Omega\subset {\mathbf R}^n$ be a bounded domain and $j, k$ be positive integers. For $1\leqslant p,~r\leqslant \infty$ and $0\leqslant j< k,$ there exists some constant $C=C(n,k,p,r,j,\theta,\Omega)$ such that
	\begin{equation}
		\|\nabla^{j}\uu\|_{L^q(\Omega)}\leqslant C\|\nabla^{k}\uu\|_{L^p(\Omega)}^{\theta} \|\uu\|_{L^r(\Omega)}^{1-\theta}
	\end{equation}
	holds for any $\uu\in W_0^{k,p}(\Omega),$ where $\frac{1}{q}=\frac{j}{n}+\theta(\frac{1}{p}-\frac{k}{n})+\frac{1-\theta}{r}.$
\end{lemma}

Next, we recall some estimates of nonlinear elliptic system stated in \cite{2023Wellposedness}.

\begin{lemma}\label{lem3}
	Let $\Omega$ be a periodic domain in ${\mathbf R}^d.$ Given a function $f\in L^p(\Omega)~(1<p<\infty)$ such that
	$$\int_\Omega f\,dx=0.$$
	Assume that $\uu$ is a unique solution to the following nonlinear elliptic system
	\begin{align}
		\begin{cases}\label{elliptic-ope-GENERAL}
			-\divv \Ss_\delta=f,\\
			\Ss_\delta =
			2\mu{\mathbb D}{\mathbf u}+\lambda \divv \uu\,\mathbb{I}+\tau^*(|{\mathbb D}{\mathbf u}|^2 +\delta^2)^{\frac{q-2}{2}} \D\uu.
		\end{cases}
	\end{align}
	Then the following assertions hold.
	\begin{itemize}
		\item If $d=1,$ then the solution $\uu$ of \eqref{elliptic-ope-GENERAL} belongs to $W^{2,p}(\Omega)$ provided that $\mu>0,\, q\geqslant 1$ and $\delta\geqslant 0$. In particular, we have
		\begin{equation*}
			\dfrac{\mu}{p} \intO |\del_x^2\uu|^p\,dx+\tau^*(q-1)\intO
			|\del_x\uu|^2(|\del_x \uu|^2+\delta^2)^{\frac{q-4}{2}}|\del_x^2 \uu|^{p}\,dx\leqslant \dfrac{\mu^{1-p}}{p} \intO |f|^p\,dx.
		\end{equation*}
		
		\item If $d=2$ (or $3$) and $f\in L^{2}(\Omega)$,  then the solution $\uu$ of \eqref{elliptic-ope-GENERAL} belongs to $H^{2}(\Omega)$ provided that $q\geqslant 1,~\delta\geqslant 0$, for the case that there exist some small $\varepsilon>0$ such that $\mu>\varepsilon,~ 2\mu+\lambda>\varepsilon.$ In particular, there exists constant $C>0$ independent of $\delta$ such that
		\begin{align*}
			&\mu \intO |\nabla \curl \uu |^2\,dx +(2\mu+\lambda) \intO |\nabla \divv \uu|^2\,dx \\
			&\quad+\tau^*\min(1,(q-1))\intO  |{\mathbb D}{\mathbf u}|^2 (|{\mathbb D}{\mathbf u}|^2+\delta^2)^{\frac{q-4}{2}} |\nabla{\mathbb D}{\mathbf u} |^2\,dx\leqslant C \intO |f|^2\,dx.
		\end{align*}
		
		\item If $d=2$ (or $3$),  then the linearized operator associated to equation \eqref{elliptic-ope-GENERAL} at a reference solution $\uu^*\in W^{2-\frac{2}{p},p}(\Omega)~( p>d+2),$ denoted by $\mathcal{A}(\uu^*, \D)$  still yield maximal $L^p-$regularity  provided that $\mu>0,\; 2\mu +\lambda>0, \; q\geqslant 1$ and $\delta \geqslant c$ for some $c>0.$ Moreover, there exists constant $C>0$ dependent on $\delta$ (the notation $C(\delta^{-1})$ concerns a constant who may have an unfavorable effect when $\delta$ becomes close to zero), such that
		\[
		\| \uu \|_{W^{2,p}(\Omega)}\leqslant \|\mathcal{A}(\uu^*, \D) \uu\|_{L^p (\Omega)}\leqslant C(\delta^{-1}) \| f\|_{L^p (\Omega)}.
		\]
	\end{itemize}
\end{lemma}

In addition, we explain the notations and conventions uesd throughout this paper. We define the function $\beta: \RR^+\rightarrow \RR^+$ as
\begin{equation}\label{def-gamma}
	\beta(s)=\mu +\dfrac{\tau^*}{2}s^{\frac{q-2}{2}},
\end{equation}
and denote by $B:= |{\mathbb D}{\mathbf u}|^2 +\delta^2$, then we can rewrite $\divv \Ss$ as
$$\divv \Ss=\beta(B)\Delta \uu +\big(\lambda+\beta(B)\big)\nabla \divv \uu +2\beta^{\prime}(B)\nabla\big( |{\mathbb D}{\mathbf u}|^2\big)\cdot {\mathbb D}{\mathbf u}.$$
Since $\D$ is symmetric, the $i^{th}$ entry of $\divv \Ss$ becomes
\begin{align*}
	\big[\divv \Ss\big]_i&= \underset{k=1}{\overset{3}{\sum}}\big(\beta(B)\del_k^2 \uu_i+(\lambda+\beta(B))\del_i\del_k \uu_k\big)+4\beta^{\prime}(B)\underset{j,k,l=1}{\overset{3}{\sum}}{\mathbb D}_{ij}{\mathbf u} {\mathbb D}_{kl}{\mathbf u}\del_j{\mathbb D}_{kl}{\mathbf u}\\
	&=\underset{k=1}{\overset{3}{\sum}}\big(\beta(B)\del_k^2 \uu_i+(\lambda+\beta(B))\del_i\del_k \uu_k\big)+4\beta^{\prime}(B)\underset{j,k,l=1}{\overset{3}{\sum}}{\mathbb D}_{ik}{\mathbf u}{\mathbb D}_{jl}{\mathbf u}\del_k\del_l \uu_{j}\\
	&=\underset{j,k,l=1}{\overset{3}{\sum}}a_{ij}^{kl}\del_k\del_l\uu_j,
\end{align*}
where
\begin{equation*}
	a_{i,j}^{k,l}=\beta(B) \delta_{kl}\delta_{ij}+(\lambda+\beta(B))\delta_{il}\delta_{jk}+4\beta^{\prime}(B){\mathbb D}_{ik}{\mathbf u}{\mathbb D}_{jl}{\mathbf u},
\end{equation*}
with $\delta_{kl}$ denoting the Kronecker symbol. Define the quasi-linear differential operator $\mathcal{A}(\uu,\D)$ as
\begin{equation}\label{defA}
	\mathcal{A}(\uu,\D)=\underset{k,l=1}{\sum}A^{k,l}{\mathbb D}_{k}{\mathbf u}{\mathbb D}_{l}{\mathbf u},
\end{equation}
where the matrix-valued coefficients
\begin{equation*}
	A^{k,l}(\uu)=\big(a_{i,j}^{k,l}\big).
\end{equation*}

\section{Proof of Theorem \ref{theo1}}\label{Sec-The1}
Let $(\rho,u)$ be a local strong solution to the problem \eqref{1dlg-E2}-\eqref{initial value} as described in Theorem \ref{theo1}.

\begin{lemma}\label{L3.1}
	Under the condition of Theorem \ref{theo1}, the standard energy estimate yields that
	\begin{equation}\label{energy}
		\sup_{0\leqslant t\leqslant T}\int_{\Omega}(\rho|\uu|^{2}+\rho^\gamma)dx+\int_{0}^{T}\int_{\Omega}|\nabla \uu|^{2}dxdt+\int_{0}^{T}\int_{\Omega}|{\mathbb D}{\mathbf u}|^{q}dxdt\leqslant C,
	\end{equation}
	holds for any $T\in (0,T^*),$ where $C$ is a constant depending on the initial date.
\end{lemma}

Suppose that there exists some constant $M>0 $ such that
\begin{equation}\label{counter}
	\int_{0}^{T}\|\nabla \uu\|_{L^{\infty}(\Omega)}^{4}dt\leqslant M,
\end{equation}
for any $0<T<T^*$.

\begin{lemma}\label{L3.2}
	Under the condition of Theorem \ref{theo1} and \eqref{counter}, there exists a positive constant $C>0$ such that
	\begin{equation}\label{2dlg-E4}
		\sup_{0\leqslant t\leqslant T}\|\rho\|_{L^{\infty}(\Omega)}\leqslant C,
	\end{equation}
	holds for any $T\in (0,T^*).$
\end{lemma}

The proof of Lemma \ref{L3.2} is simple, and the proof is omitted here.

\begin{lemma}\label{L3.3}
	Under the condition of Theorem \ref{theo1} and \eqref{counter},  there exists a positive constant $C$ such that
	\begin{equation}\label{tidu u}
		\frac{d}{dt}\|\nabla \uu\|_{L^{2}(\Omega)}^{2}+\int_{\Omega}\rho|\uu_{t}|^{2}dx\leqslant C\|\nabla \uu\|_{L^{\infty}(\Omega)}(1+\|\sqrt{\rho}\uu_{t}\|_{L^{2}(\Omega)}^{2})+\varepsilon\|\nabla \uu_{t}\|_{L^{2}(\Omega)}^{2}+C,
	\end{equation}
	for any fixed $\varepsilon\in (0,1).$
\end{lemma}

\begin{proof}
	Multiplying $\eqref{1dlg-E2}_{2}$ by $\uu_t$ and integrating the resulting equation over $\Omega$, we obtain that
	\begin{align}\label{188-ch}
		&\int_{\Omega}\rho |\uu_t|^{2}\,dx+\dfrac{1}{2}\dfrac{d}{dt}\int_{\Omega}\big(\mu |\nabla \uu|^{2}+(\lambda+\mu)(\divv \uu)^2+\dfrac{2}{q}\tau^* (|{\mathbb D}{\mathbf u}|^2 +\delta^2)^{\frac{q}{2}}\big)\,dx\nonumber\\
		&=\intO \big(-\rho\uu\cdot\nabla \uu\big)\cdot\uu_t\,dx-\intO \nabla p\cdot \uu_t\,dx.
	\end{align}
	Based on Lemma \ref{L3.1} and Lemma \ref{L3.2}, each term on the right-hand side of equation \eqref{188-ch} is estimated as follows
	\begin{align*}
		\quad\int_{\Omega}-\rho \uu\cdot\nabla \uu\cdot \uu_{t}~dx &\leqslant\|\nabla \uu\|_{L^{\infty}(\Omega)}\|\sqrt{\rho}\uu\|_{L^{2}(\Omega)}\|\sqrt{\rho}\uu_{t}\|_{L^{2}(\Omega)}\\
		&\leqslant C\|\nabla \uu\|_{L^{\infty}(\Omega)}(1+\|\sqrt{\rho}\uu_{t}\|_{L^{2}(\Omega)}^{2}),\\
		-\intO \nabla p\cdot \uu_t\,dx&=\intO p\operatorname{div}\uu_t\,dx\leqslant \|p\|_{L^{2}(\Omega)}\|\operatorname{div}\uu_{t}\|_{L^{2}(\Omega)}\leqslant \varepsilon\|\nabla \uu_{t}\|_{L^{2}(\Omega)}^{2}+C.
	\end{align*}
	for any fixed $\varepsilon\in (0,1).$	Thus, one arrives at
	\begin{align}
		&\int_{\Omega}\rho|\uu_{t}|^{2}dx+\frac{1}{2}\frac{d}{dt}\int_{\Omega}(\mu|\nabla \uu|^{2}+(\lambda+\mu)(\operatorname{div}\uu)^{2}+\dfrac{2}{q}\tau^* (|{\mathbb D}{\mathbf u}|^2 +\delta^2)^{\frac{q}{2}})dx\nonumber\\
		&\leqslant C\|\nabla \uu\|_{L^{\infty}(\Omega)}(1+\|\sqrt{\rho}\uu_{t}\|_{L^{2}(\Omega)}^{2})+\varepsilon\|\nabla \uu_{t}\|_{L^{2}(\Omega)}^{2}+C.
	\end{align}
	and so \eqref{tidu u} is obtained.
\end{proof}

\begin{lemma}\label{L3.5}
	Under the condition of Theorem \ref{theo1} and \eqref{counter},  there exists a positive constant $C$ such that
	\begin{align}\label{tidu ut}
		&\frac{d}{dt}\int_{\Omega}\left(\rho|\uu_{t}|^{2}+p(\operatorname{div}\uu)^{2}\right)dx+\int_{\Omega}|\nabla \uu_{t}|^{2}dx\\
		&\leqslant C(\|\nabla \uu\|_{L^{\infty}(\Omega)}^{4}+1)(1+\|\sqrt{\rho} \uu_{t}\|_{L^{2}(\Omega)}^{2}+\|\nabla p\|_{L^{2}(\Omega)}^{2})+\varepsilon\|\nabla^{2}\uu\|_{L^{2}(\Omega)}^{2},\nonumber
	\end{align}
	for any fixed $\varepsilon\in (0,1).$
\end{lemma}

\begin{proof}	
	The equation $\eqref{1dlg-E2}_{2}$ is rewritten as
	\begin{equation*}
		\rho \del_t\uu +(\rho \uu\cdot \nabla) \uu -\mu \Delta \uu -(\lambda+\mu)\nabla \divv \uu -\tau^* \divv\big(( |{\mathbb D}{\mathbf u}|^2 +\delta^2)^{\frac{q-2}{2}}{\mathbb D}{\mathbf u}\big)+\nabla p=0.
	\end{equation*}
	Differentiating the above equation with respect to time, one gets that
	\begin{align*}
		&\rho \uu_{tt}+\rho (\uu\cdot\nabla) \uu_t-\mu\Delta \uu_t-(\lambda+\mu)\nabla \divv \uu_t-\tau^*{\rm div}\big(( |{\mathbb D}{\mathbf u}|^2 +\delta^2)^{\frac{q-2}{2}}{\mathbb D}{\mathbf u} \big)_t+\nabla p_t\\
		&=-\rho_t(\uu_t+(\uu\cdot \nabla) \uu)-\rho (\uu_t\cdot\nabla) \uu.
	\end{align*}
	Multiplying the above equation by $\uu_t$ and integrating the resulting equation over $\Omega$, one obtains that
	\begin{align}\label{25-ch}
			&\frac{1}{2}\dfrac{d}{dt}\intO\rho|\uu_t|^2\,dx
			+\mu\intO|\nabla \uu_t|^2\,dx +(\lambda+\mu)\intO |\divv \uu_t|^2\, dx-\tau^*\int_{\Omega}\uu_t\cdot{\rm div}\Big(( |{\mathbb D}{\mathbf u}|^2 +\delta^2)^{\frac{q-2}{2}}{\mathbb D}{\mathbf u}\Big)_t\,dx\nonumber\\
			&=-\intO \rho_t \big(\uu_t+(\uu\cdot \nabla) \uu\big)\cdot\uu_t\,dx -\intO\rho (\uu_t\cdot\nabla) \uu\cdot \uu_t\,dx+\intO  p_t \divv \uu_t\,dx \nonumber\\
			& = - \intO\rho \uu\cdot\nabla((\uu_t+(\uu\cdot \nabla) \uu)\cdot \uu_t)\,dx-\intO\rho( \uu_t\cdot\nabla) \uu\cdot \uu_t\,dx+\intO  p_t \divv \uu_t\,dx .
	\end{align}
	It follows from the equation $\eqref{1dlg-E2}_{1}$ that
	\begin{equation}\label{transform}
		p_{t}+\operatorname{div}(p\uu)+(\gamma-1)p\operatorname{div}\uu=0.
	\end{equation}
	Hence,
	\begin{align*}
		-\intO p_t\divv \uu_t\,dx &= \intO (\nabla p\cdot \uu +\gamma p\divv \uu)\divv \uu_t\,dx\\
		&=\intO\left( (\nabla p\cdot \uu )\divv \uu_t\right)\,dx+\dfrac{\gamma}{2}\Dt\intO p(\divv\uu)^2\,dx-\dfrac{\gamma}{2}\intO p_t (\divv \uu)^2\,dx\\
		& = \frac{d}{dt}\intO\frac{\gamma}{2}p(\divv \uu)^2dx+\intO\nabla p\cdot \uu\divv \uu_t\,dx\\
		&\quad +\frac{\gamma}{2}\Big(\intO-p \uu\cdot\nabla(\divv \uu)^2\,dx+(\gamma-1)\intO p(\divv \uu)^3\,dx\Big).
	\end{align*}
	Substituting this identity into \eqref{25-ch}, we obtain that
	\begin{align}\label{26ch}
		\frac{d}{dt}&\intO \left(\frac{1}{2}\rho|\uu_t|^2+\frac{\gamma}{2}p (\divv \uu)^2 \right)\,dx+\mu \intO |\nabla \uu_t|^2\,dx+(\lambda+\mu)\intO|\divv \uu_t|^2\,dx\nonumber\\
		&\quad\quad-\tau^*\int_{\Omega}\uu_t\cdot{\rm div}\Big(( |{\mathbb D}{\mathbf u}|^2 +\delta^2)^{\frac{q-2}{2}}{\mathbb D}{\mathbf u}\Big)_t
		\,dx \nonumber\\
		& \leqslant C\intO\Big(p |\nabla \uu|^3+p |\uu||\nabla \uu||\nabla^2 \uu|+|\nabla p||\uu||\nabla \uu_t|+\rho |\uu||\uu_t||\nabla \uu_t|\nonumber\\
		&\quad \quad +\rho |\uu||\uu_t|| \nabla \uu|^2+ \rho |\uu|^2|\uu_t||\nabla^2 \uu|+\rho |\uu|^2|\nabla \uu|| \nabla \uu_t|+\rho |\uu_t|^2|\nabla \uu|\Big)\,dx\nonumber\\
		&:=\sum_{k=1}^{8} I_k.
	\end{align}
	Note that
	\begin{align*}
		&-\int_{\Omega}\uu_t\cdot{\rm div}\Big(( |{\mathbb D}{\mathbf u}|^2 +\delta^2)^{\frac{q-2}{2}}{\mathbb D}{\mathbf u}\Big)_t
		\,dx\\
		&=\int_{\Omega}{\mathbb D}{\mathbf u}_t:\big( ( |{\mathbb D}{\mathbf u}|^2 +\delta^2)^{\frac{q-2}{2}}{\mathbb D}{\mathbf u} \big)_t\,dx\\
		&=\int_{\Omega}( |{\mathbb D}{\mathbf u}|^2 +\delta^2)^{\frac{q-2}{2}}{\mathbb D}{\mathbf u}_t:{\mathbb D}{\mathbf u}_t\,dx+\dfrac{q-2}{2}\int_{\Omega} ( |{\mathbb D}{\mathbf u}|^2 +\delta^2)^{\frac{q-4}{2}}{\mathbb D}{\mathbf u}(|{\mathbb D}{\mathbf u}|^{2})_{t}:{\mathbb D}{\mathbf u}_{t}\,dx\\
		&=\int_{\Omega} ( |{\mathbb D}{\mathbf u}|^2 +\delta^2)^{\frac{q-2}{2}}|{\mathbb D}{\mathbf u}_t|^2\,dx+(q-2)\int_{\Omega}( |{\mathbb D}{\mathbf u}|^2 +\delta^2)^{\frac{q-4}{2}}|{\mathbb D}{\mathbf u}|^{2} |{\mathbb D}{\mathbf u}_t|^2\,dx\\
		&\geqslant \int_{\Omega} ( |{\mathbb D}{\mathbf u}|^2 +\delta^2)^{\frac{q-4}{2}}\big((q-1)|{\mathbb D}{\mathbf u}|^2+\delta^2  \big)|{\mathbb D}{\mathbf u}_t|^2\,dx\geqslant 0
	\end{align*}
	for all $q>1.$
	Now, each term on the right hand of (\ref{26ch}) is estimated as follows
	\begin{align*}
		&|I_{1}|\leqslant C \|p\|_{L^{\infty}(\Omega)}\|\nabla \uu\|_{L^{3}(\Omega)}^{3}\leqslant C\|\nabla \uu\|_{L^{\infty}(\Omega)}^{3},\\
		& |I_{2}|\leqslant C\|\rho\|_{L^{\infty}(\Omega)}^{\gamma-\frac{1}{2}}\|\sqrt{\rho }\uu\|_{L^{2}(\Omega)}\|\nabla \uu\|_{L^{\infty}(\Omega)}\|\nabla^{2}\uu\|_{L^{2}(\Omega)}\leqslant C\|\nabla \uu\|_{L^{\infty}(\Omega)}^{2}+\varepsilon\|\nabla^{2}\uu\|_{L^{2}(\Omega)}^{2},\\
		& |I_{3}|\leqslant C\|\nabla p\|_{L^{2}(\Omega)}\|\uu\|_{L^{\infty}(\Omega)}\|\nabla \uu_{t}\|_{L^{2}(\Omega)}\leqslant C\|\nabla \uu\|_{L^{\infty}(\Omega)}^{2}\|\nabla p\|_{L^{2}(\Omega)}^{2}+\varepsilon\|\nabla \uu_{t}\|_{L^{2}(\Omega)}^{2},\\
		& |I_{4}|\leqslant C\|\sqrt{\rho}\uu_{t}\|_{L^{2}(\Omega)}\|\uu\|_{L^{\infty}(\Omega)}\|\nabla \uu_{t}\|_{L^{2}(\Omega)}\leqslant C\|\nabla \uu\|_{L^{\infty}(\Omega)}^{2}\|\sqrt{\rho}\uu_{t}\|_{L^{2}(\Omega)}^{2}+\varepsilon\|\nabla \uu_{t}\|_{L^{2}(\Omega)}^{2},\\
		& |I_{5}|\leqslant C\|\sqrt{\rho}\uu\|_{L^{2}(\Omega)}\|\sqrt{\rho} \uu_{t}\|_{L^{2}(\Omega)}\|\nabla \uu\|_{L^{\infty}(\Omega)}^{2}\leqslant C\|\nabla \uu\|_{L^{\infty}(\Omega)}^{2}(\|\sqrt{\rho} \uu_{t}\|_{L^{2}(\Omega)}^{2}+1),\\
		& |I_{6}|\leqslant C\|\sqrt{\rho} \uu_{t}\|_{L^{2}(\Omega)}\|\nabla \uu\|_{L^{\infty}(\Omega)}^{2}\|\nabla^{2}\uu\|_{L^{2}(\Omega)}\leqslant C\|\nabla \uu\|_{L^{\infty}(\Omega)}^{4}\|\sqrt{\rho} \uu_{t}\|_{L^{2}(\Omega)}^{2}+\varepsilon\|\nabla^{2}\uu\|_{L^{2}(\Omega)}^{2},\\
		& |I_{7}|\leqslant C\|\sqrt{\rho} \uu\|_{L^{2}(\Omega)}\|\nabla \uu\|_{L^{\infty}(\Omega)}^{2}\|\nabla \uu_{t}\|_{L^{2}(\Omega)}\leqslant C\|\nabla \uu\|_{L^{\infty}(\Omega)}^{4}+\varepsilon\|\nabla \uu_{t}\|_{L^{2}(\Omega)}^{2},\\
		& |I_{8}|\leqslant C\|\nabla \uu\|_{L^{\infty}(\Omega)}\|\sqrt{\rho} \uu_{t}\|_{L^{2}(\Omega)}^{2}.
	\end{align*}
	Substituting all the estimates into~\eqref{26ch}, we get that
	\begin{equation*}
		\begin{aligned}
			&\frac{d}{dt}\int_{\Omega}\left(\rho|\uu_{t}|^{2}+p(\operatorname{div}\uu)^{2}\right)dx+\int_{\Omega}|\nabla \uu_{t}|^{2}dx\\
			&\leqslant C(\|\nabla \uu\|_{L^{\infty}(\Omega)}^{4}+1)(1+\|\sqrt{\rho} \uu_{t}\|_{L^{2}(\Omega)}^{2}+\|\nabla p\|_{L^{2}(\Omega)}^{2})+\varepsilon\|\nabla^{2}\uu\|_{L^{2}(\Omega)}^{2}.
		\end{aligned}
	\end{equation*}
	for any fixed $\varepsilon\in (0,1).$
\end{proof}

\begin{lemma}\label{L3.6}
	Under the condition of Theorem \ref{theo1} and \eqref{counter},  there exists a positive constant $C$ such that
	\begin{eqnarray}
		&&\frac{d}{dt}\|\nabla p\|_{L^6(\Omega)}
		\leqslant  C( \|\nabla \uu\|_{L^\infty(\Omega)}\|\nabla p\|_{L^6 (\Omega)}+\|\nabla^2 \uu\|_{L^6 (\Omega)}),\label{3dlg-E15}\\
		&&\frac{d}{dt}\|\nabla p\|_{L^2(\Omega)}^{2}
		\leqslant  C( \|\nabla \uu\|_{L^\infty(\Omega)}+1 )\|\nabla p\|_{L^2(\Omega)}^{2}+\varepsilon\|\nabla^2 \uu\|_{L^2 (\Omega)}^{2},\label{3dlg-E15-1}
	\end{eqnarray}
	and
	\begin{equation}\label{Lp-ch}
		\quad\quad\|\uu\|_{W^{2,6}(\Omega)}\leqslant C(1+\|\nabla \uu\|^{2}_{H^{1}(\Omega)}+\|\nabla p\|_{L^{6}(\Omega)})+\varepsilon\|\nabla \uu_{t}\|_{L^{2}(\Omega)}^2,
	\end{equation}
	for any fixed $\varepsilon\in (0,1).$	
\end{lemma}

\begin{proof}
	The equation $\eqref{1dlg-E2}_{1}$ implies that
	\begin{equation*}\label{cont-pressure}
		p_t+\divv(p\uu)+(\gamma-1)p\divv \uu=0.
	\end{equation*}
	Differentiating it with respect to $x_k$ yields
	\begin{equation}\label{xk}
		(p_{x_k})_t+\uu\cdot \nabla p_{x_k}+\nabla p\cdot \uu_{x_k}+\gamma p_{x_k}\divv \uu +\gamma p\divv \uu_{x_k}=0.
	\end{equation}
	Now, multiplying the equation (\ref{xk}) by $p_{x_k}|p_{x_k}|^4,$ integrating the resulting equation over $\Omega$ and summing  over $k$, we get that
	\begin{equation*}
		\begin{split}
			\frac{d}{dt}\|\nabla p\|^{6}_{L^{6}(\Omega)}
			&\leqslant  C\intO |\nabla \uu||\nabla p|^6\,dx+\intO p|\nabla \divv \uu||\nabla p|^{5}\,dx\\
			& \leqslant  C \big( \|\nabla \uu\|_{L^{\infty}(\Omega)}\|\nabla p\|_{L^6 (\Omega)}^6+\|p\|_{L^\infty(\Omega)} \|\nabla \divv \uu\|_{L^6(\Omega)}\|\nabla p\|_{L^6 (\Omega)}^{5}\big).
		\end{split}
	\end{equation*}
	So, \eqref{3dlg-E15} is deduced from Lemma \ref{L3.2}. By similar way, multiplying the equation (\ref{xk}) by $p_{x_k},$ one arrives at
	\begin{equation*}
		\frac{d}{dt}\|\nabla p\|_{L^2(\Omega)}^{2}
		\leqslant  C( \|\nabla \uu\|_{L^\infty(\Omega)}+1 )\|\nabla p\|_{L^2(\Omega)}^{2}+\varepsilon\|\nabla^2 \uu\|_{L^2 (\Omega)}^{2}.
	\end{equation*}
	
	In order to prove \eqref{Lp-ch}, one has to show the $L^p-$estimates on the elliptic operator associated to the system based on the results ststed in Lemma \ref{lem3}. Indeed, we write the momentum equation as follows
	\begin{equation}\label{lp-prob}
		-\mathcal{A}(\uu^*,\D)\uu=-\rho \uu_t-\rho (\uu\cdot\nabla \uu)-\nabla p+(\mathcal{A}(\uu,\D)\uu-\mathcal{A}(\uu^*,\D)\uu)
	\end{equation}
	where $\uu^*$ is a reference solution. Investigating \eqref{lp-prob} as a quasi-linear elliptic equation, we know that the linearized operator still yield maximal $L^p-$regularity according to the last statement in Lemma \ref{lem3}. According to the definition of the operator $\mathcal{A}$ in \eqref{defA}, we have that
	\begin{align*}
		{\mathcal A}(\uu,\D)\uu-\mathcal{A}(\uu^*,\D)\uu&=(\beta(|{\mathbb D}{\mathbf u}|^2 +\delta^2)-\beta(|{\mathbb D}{\mathbf u}^*|^2 +\delta^2))(\Delta \uu+\nabla \divv \uu) \\
		&\quad+4\underset{j,k,l=1}{\overset{3}{\sum}}(\beta^\prime(|{\mathbb D}{\mathbf u}|^2 +\delta^2){\mathbb D}_{ik}{\mathbf u}{\mathbb D}_{jl}{\mathbf u}-\beta^\prime(|{\mathbb D}{\mathbf u}^*|^2 +\delta^2){\mathbb D}_{ik}{\mathbf u}^*{\mathbb D}_{jl}{\mathbf u}^*)\del_k \del_l\uu_j,
	\end{align*}
	where $\beta(\cdot)$ was defined in \eqref{def-gamma}. Therefore, we have
	\begin{align*}\label{difA-A*-est}
		\|&\mathcal{A}(\uu,\D)\uu-\mathcal{A}(\uu^{*},\D)\uu\|_{L^{p}(\Omega)}\\
		&\leqslant C\||{\mathbb D}{\mathbf u}|^{2}-|{\mathbb D}{\mathbf u}^*|^{2}\|_{L^{\infty}(\Omega)}\|\uu\|_{W^{2,p}(\Omega)}\\
		&+C\sum_{k,l=1}^{3}\||{\mathbb D}_{ik}{\mathbf u}{\mathbb D}_{jl}{\mathbf u}-{\mathbb D}_{ik}{\mathbf u}^*{\mathbb D}_{jl}{\mathbf u}^*|\|_{L^{\infty}(\Omega)}\|\uu\|_{W^{2,p}(\Omega)}\\
		&+4C\sum_{k,l=1}^{3}\|(|{\mathbb D}{\mathbf u}|^{2}-|{\mathbb D}{\mathbf u}^*|^{2}){\mathbb D}_{ik}{\mathbf u}^*{\mathbb D}_{jl}{\mathbf u}^*\|_{L^{\infty}(\Omega)}\|\uu\|_{W^{2,p}(\Omega)}\\
		&\leqslant C\||{\mathbb D}{\mathbf u}|-|{\mathbb D}{\mathbf u}^*|\|_{L^{\infty}(\Omega)}\||{\mathbb D}{\mathbf u} |+|{\mathbb D}{\mathbf u}^*|\|_{L^{\infty}(\Omega)}\|\uu\|_{W^{2,p}(\Omega)}\\
		&+C\sum_{k,l=1}^{3}(|{\mathbb D}_{ik}{\mathbf u}|_{L^{\infty}(\Omega)}\||{\mathbb D}_{jl}{\mathbf u}|-|{\mathbb D}_{jl}{\mathbf u}^*|\|_{L^{\infty}(\Omega)}\\
		&+|{\mathbb D}_{jl}{\mathbf u}^*|_{L^{\infty}(\Omega)}\||{\mathbb D}_{ik}{\mathbf u}|-|{\mathbb D}_{ik}{\mathbf u}^*|\|_{L^{\infty}(\Omega)})\|\uu\|_{W^{2,p}(\Omega)}\\
		&+\||{\mathbb D}{\mathbf u}|-|{\mathbb D}{\mathbf u}^*|\|_{L^{\infty}(\Omega)}\|{\mathbb D}_{ik}{\mathbf u}^*{\mathbb D}_{jl}{\mathbf u}^*\|_{L^{\infty}(\Omega)}\||{\mathbb D}{\mathbf u} |+|{\mathbb D}{\mathbf u}^*|\|_{L^{\infty}(\Omega)}\|\uu\|_{W^{2,p}(\Omega)}\\
		&\leqslant C\|\nabla(\uu-\uu^{*})\|_{L^{\infty}(\Omega)}\|\uu\|_{W^{2,p}(\Omega)}.	
	\end{align*}
	Thus, one can choose the reference solution $\uu^*$ of \eqref{lp-prob} such that $\uu$ represents a small perturbation of $\uu^*$, and deduces from \eqref{lp-prob} that
	\begin{equation}\label{3dlg-E19}
		\|\uu \|_{W^{2,p}(\Omega)}\leqslant C \|-\rho \uu_t-\rho (\uu\cdot\nabla \uu)-\nabla p\|_{L^{p}(\Omega)}+\varepsilon\|\uu\|_{W^{2,p}(\Omega)}.
	\end{equation}
	So,
	\begin{equation*}
		\begin{split}
			\|\uu\|_{W^{2,6}(\Omega)}&\leqslant C(\|\rho \uu_{t}\|_{L^{6}(\Omega)}+\|\rho(\uu\cdot\nabla \uu)\|_{L^{6}(\Omega)}+\|\nabla p\|_{L^{6}(\Omega)})\\
			&\leqslant C(\|\rho\|_{L^{\infty}(\Omega)}\|\nabla \uu_{t}\|_{L^{2}(\Omega)}+\|\rho\|_{L^{\infty}(\Omega)}\|\nabla \uu\|_{L^{6}(\Omega)}^{2}+\|\nabla p\|_{L^{6}(\Omega)})\\
			&\leqslant C(\|\nabla \uu_{t}\|_{L^{2}(\Omega)}+\|\nabla \uu\|^{2}_{H^{1}(\Omega)}+\|\nabla p\|_{L^{6}(\Omega)})\\
			&\leqslant C(1+\|\nabla \uu\|^{2}_{H^{1}(\Omega)}+\|\nabla p\|_{L^{6}(\Omega)})+\varepsilon\|\nabla \uu_{t}\|_{L^{2}(\Omega)}^2.
		\end{split}
	\end{equation*}
	The proof of Lemma \ref{L3.6} is completed.
\end{proof}

\begin{lemma}\label{L3.7}
	Under the condition of Theorem \ref{theo1} and \eqref{counter},  there exists a positive constant $C$ such that
	\begin{equation}\label{GI}
		\sup_{0\leq t \leq T}\left(\|\sqrt{p}(\operatorname{div}\uu)\|_{L^{2}(\Omega)}^{2}+\|\sqrt{\rho}\uu_{t}\|_{L^{2}(\Omega)}^{2}+\|\nabla p\|_{L^{2}(\Omega)}^{2}+\|\nabla p\|_{L^{6}(\Omega)}+\|\nabla \uu\|_{L^{2}(\Omega)}^{2}\right)\leqslant C,
	\end{equation}
	and $\|\uu\|_{W^{2,6}(\Omega)}$ is almost everywhere finite on $[0,T]$ for any~$T\in (0,T^*).$
\end{lemma}

\begin{proof}
	By virtue of Lemma \ref{lem3}, we get
	\begin{align}\label{tidu2}
		\|\nabla^{2}\uu\|_{L^{2}(\Omega)}^{2}&\leqslant C(\|\rho \uu_{t}\|_{L^{2}(\Omega)}^{2}+\|\rho(\uu\cdot\nabla) \uu\|_{L^{2}(\Omega)}^{2}+\|\nabla p\|_{L^{2}(\Omega)}^{2})\nonumber\\
		&\leqslant C(\|\sqrt{\rho}\uu_{t}\|_{L^{2}(\Omega)}^{2}+\|\sqrt{\rho}\uu\|_{L^{2}(\Omega)}^{2}\|\nabla \uu\|_{L^{\infty}(\Omega)}^{2}+\|\nabla p\|_{L^{2}(\Omega)}^{2})\nonumber\\
		&\leqslant C(\|\sqrt{\rho}\uu_{t}\|_{L^{2}(\Omega)}^{2}+\|\nabla \uu\|_{L^{\infty}(\Omega)}^{2}+\|\nabla p\|_{L^{2}(\Omega)}^{2}).
	\end{align}
	So,
	\begin{equation}\label{H1}
		\|\nabla \uu\|_{H^{1}(\Omega)}^{2}\leqslant C(\|\sqrt{\rho}\uu_{t}\|_{L^{2}(\Omega)}^{2}+\|\nabla \uu\|_{L^{\infty}(\Omega)}^{2}+\|\nabla p\|_{L^{2}(\Omega)}^{2})+\|\nabla \uu\|_{L^{2}(\Omega)}^{2}.
	\end{equation}
	Gathering  \eqref{tidu u}, \eqref{tidu ut}, \eqref{3dlg-E15}, \eqref{3dlg-E15-1}, \eqref{tidu2} and \eqref{H1}, and multiplying \eqref{Lp-ch} by some appropriate coefficient, we deduce that
	\begin{align}\label{dt}
		&\frac{d}{dt}\left(\int_{\Omega}(\rho|\uu_{t}|^{2}+p(\operatorname{div}\uu)^{2})dx+\|\nabla p\|_{L^{2}(\Omega)}^{2}+\|\nabla p\|_{L^{6}(\Omega)}+\|\nabla \uu\|_{L^{2}(\Omega)}^{2}\right)\nonumber\\
		&\quad+\|\nabla \uu_{t}\|_{L^{2}(\Omega)}^{2}+\|\nabla^{2}\uu\|_{L^{2}(\Omega)}^{2}+\|\uu\|_{W^{2,6}(\Omega)}\\
		&\leqslant C(\|\nabla \uu\|_{L^{\infty}(\Omega)}^{4}+1)(1+\|\sqrt{\rho}\uu_{t}\|_{L^{2}(\Omega)}^{2}+\|\nabla p\|_{L^{2}(\Omega)}^{2}+\|\nabla p\|_{L^{6}(\Omega)}+\|\nabla \uu\|_{L^{2}(\Omega)}^{2}).\nonumber
	\end{align}
	Using Gronwall inequality and estimate \eqref{counter}, we deduce that
	\begin{eqnarray}\label{GI}
		&&\quad\|\sqrt{p}(\operatorname{div}\uu)\|_{L^{2}(\Omega)}^{2}+\|\sqrt{\rho}\uu_{t}\|_{L^{2}(\Omega)}^{2}+\|\nabla p\|_{L^{2}(\Omega)}^{2}+\|\nabla p\|_{L^{6}(\Omega)}+\|\nabla \uu\|_{L^{2}(\Omega)}^{2}\nonumber\\
		&&\leqslant C(\rho_{0},\uu_{0})exp\left(\int_{0}^{t}(1+\|\nabla \uu\|_{L^{\infty}(\Omega)}^{4})ds\right)\leqslant C.
	\end{eqnarray}
	Furthermore, one integrates \eqref{dt} over $(0,T)$ and deduces from \eqref{counter}, \eqref{Lp-ch} and \eqref{GI} that
	\begin{equation}\label{26}
		\int_0^T \|\uu\|_{W^{2,6}(\Omega)}\,dt\leqslant C.
	\end{equation}
	So, we can get
	$\|\uu\|_{W^{2,6}(\Omega)}$ is almost everywhere finite on $[0,T].$ The proof of Lemma \ref{L3.7} is completed.
\end{proof}

Note that the functions $(\rho,\uu)(x,T^{*})\triangleq\lim\limits_{t\rightarrow T^{*}}(\rho,\uu) $ satisfy
the conditions imposed on the initial data \eqref{1dlg-E7} at the time $t=T^{*}$. Furthermore,
$$-\mu\triangle \uu-(\lambda+\mu)\nabla\operatorname{div}\uu+\nabla p-\tau^{*}(( |{\mathbb D}{\mathbf u}|^2 +\delta^2)^{\frac{q-2}{2}}\mathbb{ D}\uu)\mid_{t=T^{*}}=\lim\limits_{t\rightarrow T^{*}}\rho^{\frac{1}{2}}(x,T^{*})g(x)$$
with $g(x)\triangleq\lim\limits_{t\rightarrow T^{*}}\big(\rho^{\frac{1}{2}}(\uu_{t}+\uu\cdot\nabla \uu)\big)(x,t)\in L^{2}(\Omega)$. Thus, $(\rho,\uu)(x,T^{*}) $ satisfies \eqref{1dlg-E8} also.
Now, one takes $(\rho,\uu)(x,T^{*}) $ as the initial data and applying Lemma \ref{lem1} to extend the local strong solution beyond $T^{*}.$  This is a contradiction and the proof of Theorem \ref{theo1} is completed.

\section{Proof of Theorem \ref{theo2}}\label{Sec-The2}
Let $(\rho,u)$ be a strong solution to the problem \eqref{1dlg-E2}-\eqref{initial value} as described in Theorem \ref{theo2}. Under the condition of Theorem \ref{theo1}, the standard energy estimate \eqref{energy} is holds.

Suppose that there exists some constant $M>0 $ such that
\begin{equation}\label{4dlg-E1}
	\|\rho\|_{L^{\infty}(0,T;L^{\infty}(\Omega))}+\|\nabla \uu\|_{L^{\infty}(0,T;L^{3}(\Omega))}\leqslant M,
\end{equation}
for any $0<T<T^*$.

\begin{lemma}\label{L4.1}
	Under the condition of Theorem \ref{theo2} and \eqref{4dlg-E1},  there exists a positive constant $C$ such that
	\begin{equation}\label{2tidu ut}
			\quad\frac{d}{dt}\int_{\Omega}\left(\rho|\uu_{t}|^{2}+p(\operatorname{div}\uu)^{2}\right)dx+\int_{\Omega}|\nabla \uu_{t}|^{2}\,dx\leqslant C\left(1+\|\nabla p\|_{L^{3}(\Omega)}^{2}+\|\nabla \uu\|_{H^{1}(\Omega)}^{2}+\|\sqrt{\rho} \uu_{t}\|_{L^{2}(\Omega)}^{2}\right).
	\end{equation}
\end{lemma}

\begin{proof}
	Taking same argument as \eqref{26ch}, one gets that
	\begin{align}\label{26ch-1}
		\frac{d}{dt}&\intO \left(\frac{1}{2}\rho|\uu_t|^2+\frac{\gamma}{2}p (\divv \uu)^2 \right)\,dx+\mu \intO |\nabla \uu_t|^2\,dx+(\lambda+\mu)\intO|\divv \uu_t|^2\,dx\nonumber\\
		&\quad\quad-\tau^*\int_{\Omega}\uu_t\cdot{\rm div}\Big(( |{\mathbb D}{\mathbf u}|^2 +\delta^2)^{\frac{q-2}{2}}{\mathbb D}{\mathbf u}\Big)_t
		\,dx \nonumber\\
		& \leqslant C\intO\Big(p |\nabla \uu|^3+p |\uu||\nabla \uu||\nabla^2 \uu|+|\nabla p||\uu||\nabla \uu_t|+\rho |\uu||\uu_t||\nabla \uu_t|\nonumber\\
		&\quad \quad +\rho |\uu||\uu_t|| \nabla \uu|^2+ \rho |\uu|^2|\uu_t||\nabla^2 \uu|+\rho |\uu|^2|\nabla \uu|| \nabla \uu_t|+\rho |\uu_t|^2|\nabla \uu|\Big)\,dx\nonumber\\
		&:=\sum_{k=1}^{8} I_k.
	\end{align}
	Now, each term on the right hand of (\ref{26ch}) is estimated as follows
	\begin{align*}
		&|I_{1}|\leqslant C\|\rho\|_{L^{\infty}(\Omega)}^{\gamma}\|\nabla \uu\|_{L^{3}(\Omega)}^{3}\leqslant C,\\
		& |I_{2}|\leqslant C\|p\|_{L^{\infty}(\Omega)}\| \uu\|_{L^{6}(\Omega)}\|\nabla \uu\|_{L^{3}(\Omega)}\|\nabla^{2}\uu\|_{L^{2}(\Omega)}\\
		&\hspace{17pt}\leqslant C\|\rho\|_{L^{\infty}(\Omega)}^{\gamma}\|\nabla \uu\|_{L^{2}(\Omega)}\|\nabla \uu\|_{L^{3}(\Omega)}\|\nabla^{2}\uu\|_{L^{2}(\Omega)}\\
		&\hspace{17pt}\leqslant C+C\|\nabla \uu\|_{H^{1}(\Omega)}^{2},\\
		& |I_{3}|\leqslant C\|\nabla p\|_{L^{3}(\Omega)}\|\uu\|_{L^{6}(\Omega)}\|\nabla \uu_{t}\|_{L^{2}(\Omega)}\\
		&\hspace{17pt}\leqslant C\|\nabla p\|_{L^{3}(\Omega)}\|\nabla \uu\|_{L^{2}(\Omega)}\|\nabla \uu_{t}\|_{L^{2}(\Omega)}\\
		&\hspace{17pt}\leqslant C\|\nabla p\|_{L^{3}(\Omega)}^{2}+\varepsilon\|\nabla \uu_{t}\|_{L^{2}(\Omega)}^{2},\\
		&|I_{4}|\leqslant C\|\rho\|_{L^{\infty}(\Omega)}^{\frac{3}{4}}\|\sqrt{\rho}\uu_{t}\|_{L^{2}(\Omega)}^{\frac{1}{2}}\|\uu\|_{L^{6}(\Omega)}\|\nabla \uu_{t}\|_{L^{2}(\Omega)}\|\uu_{t}\|_{L^{6}(\Omega)}^{\frac{1}{2}}\\
		&\hspace{17pt}\leqslant C\|\nabla \uu\|_{L^{2}(\Omega)}\|\sqrt{\rho}\uu_{t}\|_{L^{2}(\Omega)}^{\frac{1}{2}}\|\nabla \uu_{t}\|_{L^{2}(\Omega)}^{\frac{3}{2}}\\
		&\hspace{17pt}\leqslant C\|\sqrt{\rho}\uu_{t}\|_{L^{2}(\Omega)}^{2}+\varepsilon\|\nabla \uu_{t}\|_{L^{2}(\Omega)}^{2},\\
		& |I_{5}|\leqslant C\|\rho\|_{L^{\infty}(\Omega)}\|\uu\|_{L^{6}(\Omega)}\| \uu_{t}\|_{L^{6}(\Omega)}\|\nabla \uu\|_{L^{3}(\Omega)}^{2}\\
		&\hspace{17pt}\leqslant C\|\nabla \uu\|_{L^{2}(\Omega)}\|\nabla \uu_{t}\|_{L^{2}(\Omega)}\|\nabla \uu\|_{L^{3}(\Omega)}^{2}\leqslant C+\varepsilon\|\nabla \uu_{t}\|_{L^{2}(\Omega)}^{2},\\
		&|I_{6}|\leqslant C\|\rho\|_{L^{\infty}(\Omega)}\|\uu\|_{L^{6}(\Omega)}^{2}\| \uu_{t}\|_{L^{6}(\Omega)}\|\nabla^{2}\uu\|_{L^{2}(\Omega)}\\
		&\hspace{17pt}\leqslant C\|\nabla \uu\|_{L^{2}(\Omega)}^{2}\|\nabla \uu_{t}\|_{L^{2}(\Omega)}\|\nabla \uu\|_{H^{1}(\Omega)}\\
		&\hspace{17pt}\leqslant C\|\nabla \uu\|_{H^{1}(\Omega)}^{2}+\varepsilon\|\nabla \uu_{t}\|_{L^{2}(\Omega)}^{2},\\
		&|I_{7}|\leqslant C\|\rho\|_{L^{\infty}(\Omega)}\|\uu\|_{L^{6}(\Omega)}^{2}\|\nabla \uu\|_{L^{6}(\Omega)}\|\nabla \uu_{t}\|_{L^{2}(\Omega)}\\
		&\hspace{17pt}\leqslant C\|\nabla \uu\|_{L^{2}(\Omega)}^{2}\|\nabla \uu\|_{H^{1}(\Omega)}\|\nabla \uu_{t}\|_{L^{2}(\Omega)}\\
		&\hspace{17pt}\leqslant C\|\nabla \uu\|_{H^{1}(\Omega)}^{2}+\varepsilon\|\nabla \uu_{t}\|_{L^{2}(\Omega)}^{2},\\
		&|I_{8}|\leqslant C\|\rho\|_{L^{\infty}(\Omega)}^{\frac{1}{2}}\|\uu_{t}\|_{L^{6}(\Omega)}\|\sqrt{\rho} \uu_{t}\|_{L^{2}(\Omega)}\|\nabla \uu\|_{L^{3}(\Omega)}\\
		&\hspace{17pt}\leqslant C\|\nabla \uu_{t}\|_{L^{2}(\Omega)}\|\sqrt{\rho} \uu_{t}\|_{L^{2}(\Omega)}\\
		&\hspace{17pt}\leqslant C\|\sqrt{\rho} \uu_{t}\|_{L^{2}(\Omega)}^{2}+\varepsilon\|\nabla \uu_{t}\|_{L^{2}(\Omega)}^{2}.
	\end{align*}
	for any fixed $\varepsilon\in (0,1).$ So, it is deduced from \eqref{26ch-1} that
	\begin{equation*}
		\frac{d}{dt}\int_{\Omega}\left(\rho|\uu_{t}|^{2}+p(\operatorname{div}\uu)^{2}\right)dx+\int_{\Omega}|\nabla \uu_{t}|^{2}dx\leqslant C\left(1+\|\nabla p\|_{L^{3}(\Omega)}^{2}+\|\nabla \uu\|_{H^{1}(\Omega)}^{2}+\|\sqrt{\rho} \uu_{t}\|_{L^{2}(\Omega)}^{2}\right).
	\end{equation*}
	This completes the proof of Lemma \ref{L4.1}.
\end{proof}

\begin{lemma}\label{L4.2}
	Under the condition of Theorem \ref{theo2} and \eqref{4dlg-E1},  there exists a positive constant $C$ such that
	\begin{equation}\label{2tidu2u}
		\|\nabla^{2}\uu\|_{L^{2}(\Omega)}^{2}\leqslant C(\|\sqrt{\rho}\uu_{t}\|_{L^{2}(\Omega)}^{2}+\|\nabla p\|_{L^{2}(\Omega)}^{2}+1).
	\end{equation}
\end{lemma}

\begin{proof}
	By virtue of Lemma \ref{lem3}, one gets that
	\begin{equation*}
		\begin{aligned}
			\|\nabla^{2}\uu\|_{L^{2}(\Omega)}^{2}&\leqslant C(\|\rho \uu_{t}\|_{L^{2}(\Omega)}^{2}+\|\rho(\uu\cdot\nabla \uu)\|_{L^{2}(\Omega)}^{2}+\|\nabla p\|_{L^{2}(\Omega)}^{2})\\
			&\leqslant C(\|\sqrt{\rho}\uu_{t}\|_{L^{2}(\Omega)}^{2}+\|\uu\|_{L^{6}(\Omega)}^{2}\|\nabla \uu\|_{L^{3}(\Omega)}^{2}+\|\nabla p\|_{L^{2}(\Omega)}^{2})\\
			&\leqslant C(\|\sqrt{\rho}\uu_{t}\|_{L^{2}(\Omega)}^{2}+\|\nabla p\|_{L^{2}(\Omega)}^{2}+1).
		\end{aligned}
	\end{equation*}
	Hence
	\begin{equation}\label{2H}
		\|\nabla \uu\|_{H^{1}(\Omega)}^{2}\leqslant C(\|\sqrt{\rho}\uu_{t}\|_{L^{2}(\Omega)}^{2}+\|\nabla p\|_{L^{2}(\Omega)}^{2}+1).
	\end{equation}
	The proof of Lemma \ref{L4.2} is completed.
\end{proof}

\begin{lemma} \label{L4.3}
	Under the condition of Theorem \ref{theo2} and \eqref{4dlg-E1},  there exists a positive constant $C$ such that
	\begin{equation}\label{bound}
		\sup_{0\leq t \leq T}\left(\|\sqrt{\rho}\uu_{t}\|_{L^{2}(\Omega)}^{2}+\|\nabla p\|_{L^{6}(\Omega)}+\|\sqrt{p}(\operatorname{div}\uu)\|_{L^{2}(\Omega)}^{2}\right)\leqslant C,
	\end{equation}
	and $\|\uu\|_{W^{2,6}(\Omega)}$ is almost everywhere finite on $[0,T]$, for any $T\in(0,\min\{\frac{1}{3Cf^3(0)},T^*\}),$
	with $f(0)\triangleq 1+\|\sqrt{\rho_{0}}(\uu_{0})_{t}\|_{L^{2}(\Omega)}^{2}+\|\nabla p_{0}\|_{L^{6}(\Omega)}+\|\sqrt{p_{0}}(\operatorname{div}\uu_{0})\|_{L^{2}(\Omega)}^{2}.$
\end{lemma}

\begin{proof}
	Taking similar argument in Lemma \ref{L3.6}, gathering \eqref{2tidu ut}, \eqref{2tidu2u} and \eqref{2H}, and multiplying \eqref{Lp-ch} by some appropriate coefficient, one obtains that
	\begin{align}\label{u26}
		&\frac{d}{dt}\left(\int_{\Omega}(\rho|\uu_{t}|^{2}+p(\operatorname{div}\uu)^{2})dx+\|\nabla p\|_{L^{6}(\Omega)}\right)+\|\nabla \uu_{t}\|_{L^{2}(\Omega)}^{2}+\|\nabla^{2}\uu\|_{L^{2}(\Omega)}^{2}+\|\uu\|_{W^{2,6}(\Omega)}\nonumber\\
		&\leqslant C\left(1+\|\sqrt{\rho}\uu_{t}\|_{L^{2}(\Omega)}^{2}+\|\nabla p\|_{L^{6}(\Omega)}^{4}+\|\nabla \uu\|_{L^{\infty}(\Omega)}\|\nabla p\|_{L^{6}(\Omega)}\right)\nonumber\\
		&\leqslant C\left(1+\|\sqrt{\rho}\uu_{t}\|_{L^{2}(\Omega)}^{2}+\|\nabla p\|_{L^{6}(\Omega)}^{4}+\|\uu\|_{L^{6}(\Omega)}^{\frac{1}{4}}\|\nabla^{2}\uu\|_{L^{6}(\Omega)}^{\frac{3}{4}}\|\nabla p\|_{L^{6}(\Omega)}\right)\\
		&\leqslant C\left(1+\|\sqrt{\rho}\uu_{t}\|_{L^{2}(\Omega)}^{2}+\|\nabla p\|_{L^{6}(\Omega)}\right)^{4}+\varepsilon\|\nabla^{2}\uu\|_{L^{6}(\Omega)}\nonumber, 		
	\end{align}
	holds for any fixed $\varepsilon\in(0,1).$
	Hence,
	\begin{equation}\label{2ddt}
		\begin{aligned}
			&\frac{d}{dt}\left(\int_{\Omega}(\rho|\uu_{t}|^{2}+p(\operatorname{div}\uu)^{2})dx+\|\nabla p\|_{L^{6}(\Omega)}\right)\\
			&\leqslant C\left(1+\|\sqrt{\rho}\uu_{t}\|_{L^{2}(\Omega)}^{2}+\|\nabla p\|_{L^{6}(\Omega)}+\|\sqrt{p}(\operatorname{div}\uu)\|_{L^{2}(\Omega)}^{2}\right)^{4}.
		\end{aligned}
	\end{equation}
	Setting $f(t)\triangleq 1+\|\sqrt{\rho}\uu_{t}\|_{L^{2}(\Omega)}^{2}+\|\nabla p\|_{L^{6}(\Omega)}+\|\sqrt{p}(\operatorname{div}\uu)\|_{L^{2}(\Omega)}^{2}$, \eqref{2ddt} can rewritten as
	\begin{equation*}
		f^\prime(t)\leqslant Cf^4(t).
	\end{equation*}
	So,
	$$f(t)\leqslant f(0)(1-3Ctf^3(0))^{-\frac{1}{3}}$$
	holds for any $t\in(0,\min\{\frac{1}{3Cf^3(0)},T^*\}),$ with $f(0)\triangleq 1+\|\sqrt{\rho_{0}}(\uu_{0})_{t}\|_{L^{2}(\Omega)}^{2}+\|\nabla p_{0}\|_{L^{6}(\Omega)}+\|\sqrt{p_{0}}(\operatorname{div}\uu_{0})\|_{L^{2}(\Omega)}^{2}.$
	Thus, one integrates \eqref{u26} over $(0,T)$ to deduce that
	\begin{equation}\label{26}
		\int_0^T \|\uu\|_{W^{2,6}(\Omega)}\,dt\leqslant C
	\end{equation}
	holds for any $T\in(0,\min\{\frac{1}{3Cf^3(0)},T^*\}).$
	So, one gets that $\|\uu\|_{W^{2,6}(\Omega)}$ is almost everywhere finite on $[0,T]$ for any $T\in(0,\min\{\frac{1}{3Cf^3(0)},T^*\}).$
	The proof of Lemma \ref{L4.3} is completed.
\end{proof}

On the basis of the lemma \ref{L4.3},  taking $t_1\in(0,\min\{\frac{1}{3Cf^3(0)},T^*\}),$  one finds that                                                        	
$(\rho,\uu)(x,t_1)\triangleq\lim\limits_{t\rightarrow t_1}(\rho,\uu) $ satisfy the conditions imposed on the initial data \eqref{1dlg-E7}
at the time $t=t_1$. Furthermore,
$$-\mu\triangle \uu-(\lambda+\mu)\nabla\operatorname{div}\uu+\nabla p-\tau^{*}(( |{\mathbb D}{\mathbf u}|^2
+\delta^2)^{\frac{q-2}{2}}\operatorname{D}(\uu))\mid_{t=t_1}=\lim\limits_{t\rightarrow t_{1}}\rho^{\frac{1}{2}}(x,t_1)g(x),$$
with $g(x)\triangleq\lim\limits_{t\rightarrow t_1}\big(\rho^{\frac{1}{2}}(\uu_{t}+\uu\cdot\nabla \uu)\big)(x,t_1)\in L^{2}(\Omega)$. Thus, $(\rho,\uu)(x,t_1) $ satisfies \eqref{1dlg-E8} also.
Now, one takes $(\rho,\uu)(x,t_1) $ as the initial data and apply Lemma \ref{lem1} to extend the local strong solution beyond $T^{*}.$  This is a
contradiction and the proof of Theorem \ref{theo2} is completed.








\end{document}